\documentclass[[a4paper,UKenglish,cleveref, autoref, thm-restate]{lipics-v2021}
\usepackage{xcolor}
\usepackage{extarrows}
\usepackage{amsthm}
\usepackage{amsmath}
\usepackage{amssymb}
\usepackage{amsfonts}
\usepackage{graphicx}
\usepackage[ruled]{algorithm2e}
\usepackage{thm-restate}
\usepackage{mathtools}

\allowdisplaybreaks
\nolinenumbers

\newtheorem{thrm}{Theorem}[section]
\newtheorem{lem}[thrm]{Lemma}
\newtheorem{prop}[thrm]{Proposition}

\theoremstyle{definition}
\newtheorem{defn}[thrm]{Definition}
\theoremstyle{definition}

\theoremstyle{definition}

\theoremstyle{definition}
\newtheorem{rmk}[thrm]{Remark}

\newcommand{\Z}{\mathbb{Z}}
\newcommand{\N}{\mathbb{N}}
\newcommand{\Q}{\mathbb{Q}}
\newcommand{\Rpp}{\mathbb{R}_{> 0}}
\newcommand{\Rp}{\mathbb{R}_{\geq 0}}

\newcommand{\R}{\mathbb{R}}
\newcommand{\C}{\mathbb{C}}

\newcommand{\Zpp}{\mathbb{Z}_{>0}}
\newcommand{\Np}{\mathbb{N}_{>0}}
\newcommand{\mG}{\mathcal{G}}

\newcommand{\W}{\operatorname{W}}
\newcommand{\U}{\operatorname{U}}

\newcommand{\Coef}{\operatorname{Coef}}
\newcommand{\oB}{\overline{B}}

\title{Solving homogeneous linear equations over polynomial semirings}
\author{Ruiwen Dong}{Department of Computer Science, University of Oxford}{ruiwen.dong@kellogg.ox.ac.uk}{}{}

\authorrunning{R Dong}

\Copyright{Ruiwen Dong} 

\ccsdesc[500]{Computing methodologies~Symbolic and algebraic manipulation} 

\keywords{wreath product, identity problem, polynomial semiring, positive polynomial} 

\category{} 

\supplement{}

\acknowledgements{The author would like to thank Markus Schweighofer for useful discussions and feedback and for pointing out the references~\cite{prestel2007lectures}~and~\cite{prestel2013positive}.}

\begin{document}

\maketitle
\begin{abstract}
    For a subset $B$ of $\R$, denote by $\U(B)$ be the semiring of (univariate) polynomials in $\R[X]$ that are strictly positive on $B$.
    Let $\N[X]$ be the semiring of (univariate) polynomials with non-negative integer coefficients.
    We study solutions of homogeneous linear equations over the polynomial semirings $\U(B)$ and $\N[X]$.
    In particular, we prove local-global principles for solving single homogeneous linear equations over these semirings.
    We then show PTIME decidability of determining the existence of non-zero solutions over $\N[X]$ of single homogeneous linear equations.

    Our study of these polynomial semirings is largely motivated by several semigroup algorithmic problems in the wreath product $\Z \wr \Z$.
    As an application of our results, we show that the Identity Problem (whether a given semigroup contains the neutral element?) and the Group Problem (whether a given semigroup is a group?) for finitely generated sub-semigroups of the wreath product $\Z \wr \Z$ is decidable when elements of the semigroup generator have the form $(y, \pm 1)$.
\end{abstract}

\newpage

\section{Introduction}
Linear equations over semirings appear in various domains in mathematics and computer science, such as automata theory, optimization, and algebra of formal processes~\cite{ashby1956automata, baccelli1992synchronization, bergstra1984algebra, eilenberg1974automata}.
There have been numerous studies on linear equations over different semirings~\cite{golan2013semirings}, for example the semiring of natural numbers (integer programming), tropical semirings~\cite{pin1998tropical} and polynomial semirings~\cite{dale1976monic, narendran1996solving}.
Given a semiring $S$, define $S[X]$ to be the set of polynomials in variable $X$ whose coefficients are elements of $S$.
The set $S[X]$ is again a semiring.
One of the simplest polynomial semirings is the semiring $\N[X]$ of single variable polynomials with non-negative integer coefficients.
The problem of solving a system of linear equations over $\N[X]$ was shown to be undecidable by Narendran~\cite{narendran1996solving} using a reduction from Hilbert's tenth problem.
More precisely, given integer polynomials $h_{ij}, g_j \in \Z[X], i = 1, \ldots, n, j = 1, \ldots, k$, it is undecidable whether the system of equations
\begin{equation}\label{eq:undeceq}
    f_1 h_{1j} + \cdots + f_n h_{nj} = g_j, \quad j = 1, \ldots, k,
\end{equation}
has a solution $(f_1, \ldots, f_n)$ over $\N[X]$.
This contrasts with the decidability of solving systems of linear equations over $\N$ and over $\Z[X]$ (using respectively integer programming~\cite{kandri1988computing} and Smith canonical forms~\cite{kannan1985solving}).

In this paper, we show a decidability result for finding a non-zero solution of a \emph{single homogeneous} linear equation over $\N[X]$.
In particular, we are concerned with the following problem:
given integer polynomials $h_1, \ldots, h_n \in \Z[X]$, does the equation 
\begin{equation}\label{eq:maineq}
    f_1 h_1 + \cdots + f_n h_n = 0
\end{equation}
admit a solution $(f_1, \ldots, f_n)$ over $\N[X] \setminus \{0\}$ (i.e.\ non of the $f_i$ is zero)?

In Section~\ref{sec:dec} of this paper we give a PTIME algorithm that decides this problem.
Our algorithm relies on a local-global principle which we prove in Section~\ref{sec:locglob}, and reduces the decision problem to the existential theory of the reals in one variable.
Formal definitions of these results will be given in Section~\ref{sec:mainres}.

It turns out that the problem of solving linear equations over the semiring $\N[X]$ is closely related to solving the same equation over the semiring $\U(B)$, consisting of polynomials in $\R[X]$ that are \emph{strictly positive} on a subset $B$ of $\R$.
It is also related to the semiring $\W(B)$ of polynomials that are \emph{non-negative} on $B$.
The characterization of polynomials in $\U(B)$ and $\W(B)$ is a central subject in the theory of real algebra.
In particular, when $B$ is a semialgebraic set, variants of the \emph{positivstellensatz} give explicit descriptions of the semirings $\U(B)$ and $\W(B)$.
This theory can be traced back to the celebrated Hilbert's seventeenth problem:
given a polynomial that takes only non-negative values over the reals, can it be represented as a sum of squares of rational functions?
This has been answered positively by Artin~\cite{artin1927zerlegung} using a model theoretic approach.
The techniques proposed by Artin have since developed into the rich theory of real algebra;
for a comprehensive account of this subject, see~\cite{prestel2007lectures} or~\cite{prestel2013positive}.
An important result in solving homogeneous linear equations over $\W(\R)$ is the Br\"ocker-Prestel's local-global principle for weak isotropy of quadratic forms~\cite[Theorem~8.12, 8.13]{prestel2007lectures}.
Applied over the function field $\R(X)$, the Br\"ocker-Prestel local-global principle relates the existence of non-trivial solutions over \emph{sums of squares} in $\R(X)$ (and hence over $\W(\R)$) of a homogeneous linear equation, to the behaviour of the equation in all \emph{Henselizations} of $\R(X)$.
In Section~\ref{sec:pospoly} of this paper we prove a ``strictly positive'' version of the Br\"ocker-Prestel local-global principle, which characterizes the existence of solutions over $\U(B)$.
This will serve as a base for proving further results in Section~\ref{sec:locglob} and \ref{sec:dec}.
Our proof is inspired by Prestel's proof of the original theorem. However, several new ideas are introduced to deal with the strict positivity as well as the positivity constraint over a subset of $\R$.

An important motivation for studying linear equations over $\N[X]$ comes from a semigroup algorithmic problem in the \emph{wreath product} $\Z \wr \Z$.
The wreath product is a fundamental construction in group and semigroup theory.
Given two groups $G$ and $H$, their wreath product $G \wr H$ is defined in the following way. 
Let $G^H$ be the set of all functions $y \colon H \rightarrow G$ with finite support; it is a group with respect to pointwise multiplication.
The group $H$ acts on $G^H$ as a group of automorphisms: if $h \in H, y \in G^H$, then $y^h(b) = y(b h^{-1})$ for all $b \in H$. 
The wreath product $G \wr H$ is then defined as the semi-direct product $G^H \rtimes H$, that is, the set of all pairs $(y, h)$ where $y \in G^H, h \in H$, with multiplication operation given by
\[
    (y, h) (z, k) = (y^{k} z, hk).
\]
One easy way to understand the group $\Z \wr \Z$ is through the its isomorphism to a matrix group over the Laurent polynomial ring $\Z[X, X^{-1}]$~\cite{magnus1939theorem}:
\begin{align}\label{eq:defphi}
    \varphi \colon \Z \wr \Z & \xrightarrow{\sim} \left\{ 
\begin{pmatrix}
        1 & f \\
        0 & X^{b}
\end{pmatrix}
\;\middle|\; f \in \Z[X, X^{-1}], b \in \Z 
\right\},
\quad
(y, b) \mapsto 
    \begin{pmatrix}
        1 & \sum_{k \in \Z} y(k) X^k \\
        0 & X^{b}
    \end{pmatrix}.
\end{align}

A large number of important groups are constructed using the wreath product, such as the lamplighter group $\Z_2 \wr \Z$~\cite{grigorchuk2001lamplighter} and groups resulting from the Magnus embedding theorem~\cite{magnus1939theorem}.
The wreath product also plays an important role in the algebraic theory of automata.
The Krohn–Rhodes theorem states that every finite semigroup (and correspondingly, every finite automaton) can be decomposed into elementary components using wreath products~\cite{krohn1965algebraic}.

In Section~\ref{sec:appwr} we give an application of our results to the \emph{Identity Problem} in $\Z \wr \Z$.
Given a finite set of elements $\mG = \{A_1, \ldots, A_k\}$ in a group $G$ as well as a target element $A \in G$, denote by $\langle \mG \rangle$ the semigroup of generated by $\mG$, and by $\langle \mG \rangle_{grp}$ the group generated by $\mG$.
Consider the following decision problems:
\begin{enumerate}[(i)]
    \item \emph{(Group Membership Problem)} whether $A \in \langle \mG \rangle_{grp}$?
    \item \emph{(Semigroup Membership Problem)} whether $A \in \langle \mG \rangle$?
    \item \emph{(Identity Problem)} whether the neutral element $I$ of $G$ is contained in $\langle \mG \rangle$?
\end{enumerate}

All three problems remain undecidable even when the ambient group $G$ is restricted to relatively simple groups, such as the direct product $F_2 \times F_2$ of two free groups over two generators~\cite{bell2010undecidability, mikhailova1966occurrence}.
Indeed, one of the first undecidability results in algorithmic theory is the undecidability of the Semigroup Membership Problem for integer matrices, obtained by Markov~\cite{markov1947certain}.
Some decidability results for the Identity Problem include its NP-completeness in $\mathsf{SL}(2, \Z)$~\cite{bell2017identity} and its PTIME decidability in nilpotent groups of class at most ten~\cite{https://doi.org/10.48550/arxiv.2208.02164}.

Let $p \in \Zpp$.
The group $\Z \wr \Z$ shares some common properties with the wreath product $\left(\Z / p \Z\right) \wr \Z$ and with the Baumslag-Solitar group $\mathsf{BS}(1, p)$.
Similar to the isomorphism~\eqref{eq:defphi}, both $\left(\Z / p \Z\right) \wr \Z$ and $\mathsf{BS}(1, p)$ can be represented as $2 \times 2$ upper triangular matrix groups:
\begin{align*}
\left(\Z / p \Z\right) \wr \Z & \cong \left\{ 
\begin{pmatrix}
        1 & f \\
        0 & X^{b}
\end{pmatrix}
\;\middle|\; f \in \left(\Z / p \Z\right)[X, X^{-1}], b \in \Z 
\right\},
\\
\mathsf{BS}(1, p) & \cong \left\{ 
\begin{pmatrix}
        1 & f \\
        0 & p^{b}
\end{pmatrix}
\;\middle|\; f \in \Z[1/p], b \in \Z 
\right\}.
\end{align*}
Lohrey, Steinberg and Zetzsche showed decidability of the \emph{Rational Subset Membership Problem} (which subsumes all three decision problems) in $H \wr V$, where $H$ is a finite and $V$ is virtually free~\cite{lohrey2015rational}.
This notably implies its decidability in $\left(\Z / p \Z\right) \wr \Z$.
Cadilhac, Chistikov and Zetzsche proved its decidability in $\mathsf{BS}(1, p)$~\cite{DBLP:conf/icalp/CadilhacCZ20}.
For $\Z \wr \Z$, decision problems are much harder due to higher encoding power of the ring $\Z[X, X^{-1}]$.
The Group Membership Problem in $\Z \wr \Z$ can be reduced to the membership problem for modules over the $\Z[X, X^{-1}]$, and is hence decidable~\cite{romanovskii1974some}.
As for the Semigroup Membership Problem in $\Z \wr \Z$, Lohrey \emph{et al.} showed its undecidability using an encoding of 2-counter machines~\cite{lohrey2015rational}.
Decidability of the Identity Problem in $\Z \wr \Z$ remains an intricate open problem.
In this paper we give a decidability result in the case where all the elements of the generator $\mG$ are of the form $(y, \pm 1)$. 

\section{Main results}\label{sec:mainres}
In this section we sum up the main results of this paper.
For a subset $B$ of $\R$, denote by $\U(B)$ the set of polynomials in $\R[X]$ that are strictly positive on $B$:
\[
    \U(B) \coloneqq \{f \in \R[X] \mid f(x) > 0 \text{ for all } x \in B\}.
\]
Define $\oB$ to be the closure of $B$ in $\R$ under the Euclidean topology.
Our first result is a local-global principle for solutions of homogeneous linear equations over $\U(B)$.
Theorem~\ref{thm:locglobpos} will be proved in Section~\ref{sec:pospoly}.
\begin{restatable}{thrm}{locglobpos}\label{thm:locglobpos}
    Given polynomials $h_1, \ldots, h_n \in \R[X]$ that satisfy $\gcd(h_1, \ldots, h_n) = 1$, let $B$ be a subset of $\R$.
    If the equation $f_1 h_1 + \cdots + f_n h_n = 0$ has no solution $(f_1, \ldots, f_n)$ over $\U(B)$, then there exists a real number $t \in \overline{B}$, such that the values $h_i(t), i = 1, \ldots, n$ are either all non-negative or all non-positive.
\end{restatable}
Our second result is a corollary of the previous theorem, it provides a similar local-global principle for solutions over $\N[X] \setminus \{0\}$.
Theorem~\ref{thm:mainlocglob} will be proved in Section~\ref{sec:locglob}.

\begin{restatable}{thrm}{mainlocglob}\label{thm:mainlocglob}
    Given polynomials $h_1, \ldots, h_n \in \Z[X]$ with $\gcd(h_1, \ldots, h_n) = 1$.
    If the equation $f_1 h_1 + \cdots + f_n h_n = 0$ has no solution $(f_1, \ldots, f_n)$ over $\N[X] \setminus \{0\}$, then there exists $t \in \Rp$, such that the values $h_i(t), i = 1, \ldots, n$ are either all non-negative or all non-positive.
\end{restatable}

Our next result shows that it is decidable in PTIME whether a linear homogeneous equation is solvable over $\N[X] \setminus \{0\}$.
The input size is defined as the total number of bits used to encode all the coefficients of all $h_{i}$.
Theorem~\ref{thm:decidability} will be proved in Section~\ref{sec:dec}.

\begin{restatable}{thrm}{decidability}\label{thm:decidability}
    Given as input $h_1, \ldots, h_n \in \Z[X]$. 
    It is decidable in polynomial time whether the equation $f_1 h_1 + \cdots + f_n h_n = 0$ has solutions $f_1, \ldots, f_n$ over $\N[X] \setminus \{0\}$.
\end{restatable}

An application of this theorem is the following partial decidability result on the Identity Problem in the wreath product $\Z \wr \Z$.
This will be the main topic of Section~\ref{sec:appwr}.
\begin{restatable}{thrm}{wreathdec}\label{thm:wreathdec}
    Given a finite set of elements $\mG = \{(y_1, b_1), \ldots, (y_n, b_n)\}$ in $\Z \wr \Z$, where $b_i = \pm 1$ for all $i$.
    The following are decidable:
    \begin{enumerate}
        \item (Group Problem) whether the semigroup $\langle \mG \rangle$ generated by $\mG$ is a group.
        \item (Identity Problem) whether the neutral element $I$ is in the semigroup $\langle \mG \rangle$.
    \end{enumerate}
\end{restatable}

\section{Preliminaries}\label{sec:prelim}
In this section we introduce the necessary mathematical tools on (semi)orderings of fields as well as valuations.
Most notations and definitions follow those given in Prestel's book~\cite{prestel2007lectures}.

\subsection{Orderings and semiorderings}

\begin{defn}[Ordering]
A linear ordering of a set $S$ is a binary relation that satisfies 
\begin{enumerate}[(i)]
    \item $a \leq a$,
    \item $a \leq b, b \leq c \implies a \leq c$,
    \item $a \leq b, b \leq a \implies a = a$,
    \item $a \leq b$ or $b \leq a$
\end{enumerate}
for all $a, b, c \in S$.

Given a field $F$ of characteristic zero, a \emph{(field) ordering} of $F$ is a linear ordering $\leq$ of the underlying set of $F$ that additionally satisfies
\begin{enumerate}[(i)]
    \item $a \leq b \implies a + c \leq b + c$,
    \item $0 \leq a, 0 \leq b \implies 0 \leq ab$
\end{enumerate}
for all $a, b, c \in F$.
A field is called \emph{formally real} if it admits at least one ordering.
\end{defn}

The \emph{semiordering} of a field, defined below, is a weaker version of the field ordering.
\begin{defn}[Semiordering]
A \emph{semiordering} of a field $F$ is a linear ordering $\leq$ of the underlying set of $F$ that satisfies
\begin{enumerate}[(i)]
    \item $a \leq b \implies a + c \leq b + c$,
    \item $0 \leq 1$,
    \item $0 \leq a \implies 0 \leq ab^2$
\end{enumerate}
for all $a, b, c \in F$.
\end{defn}
In a field $F$ with semiordering $\leq$, we have $0 \leq x^2$ for all $x \in F$.
The field of real numbers $\R$ hence admits a unique semiordering, since every positive real can be written as a square.
This semiordering is simply the natural ordering on $\R$.

It is easy to see that an ordering is always a semiordering.
Conversely, a semiordering need not be an ordering.
However, in any field, the existence of a semiordering implies that of an ordering.
\begin{lem}[{\cite[Corollary~1.15]{prestel2007lectures}}]\label{lem:115}
    A field $F$ is formally real (admits an ordering) if and only if it admits a semiordering.
\end{lem}

For a semiordering $\leq$ of $F$, the set $P \coloneqq \{a \in F \mid 0 \leq a\}$ satisfies
\begin{enumerate}[(i)]
    \item $P + P \subseteq P$,
    \item $F^2 \cdot P \subseteq P$ and $1 \in P$,
    \item $P \cap -P = \{0\}$,
    \item $P \cup -P = F$.
\end{enumerate}
Such a set will be called a \emph{semicone} of $F$.
A semicone $P$ of $F$ determines a semiordering $\leq$ of $F$ by $a \leq b \iff b - a \in P$.
Therefore, we will sometimes call $P$ a semiordering as well.

The \emph{pre-semicone} is yet a weaker version of the semiordering (or semicone).
\begin{defn}[Pre-semicone]\label{def:presemicone}
    A \emph{pre-semicone} of a field $F$ is a subset $P$ of $F$ that satisfies
    \begin{enumerate}[(i)]
        \item $P + P \subseteq P$,
        \item $F^2 \cdot P \subseteq P$,
        \item $P \cap -P = \{0\}$.
    \end{enumerate}
\end{defn}
The only difference between a pre-semicone and a semicone is the absence of the rule~(iv) and the condition $1 \in P$ in (ii).
Obviously every semicone is also a pre-semicone.
Conversely, a pre-semicone need not be a semicone, but it can always be extended to one.

\begin{lem}[{\cite[Lemma~1.13]{prestel2007lectures}}]\label{lem:113}
    For every pre-semicone $P_0$ of a formally real field $F$ there exists a set $P \supseteq P_0$ such that $P$ or $-P$ is a semicone of $F$.
\end{lem}

Suppose $F$ is of characteristic zero.
A semiordering or an ordering $\leq$ of $F$ is called \emph{archemedean} if for each $a \in F$ there exists $n \in \N \subseteq F$ such that $a \leq n$.

\begin{lem}[{\cite[Lemma~1.20]{prestel2007lectures}}]\label{lem:120}
    Every archimedean semiordering is an ordering.
\end{lem}

\subsection{Valuations}\label{subsec:val}

Let $F$ be a field.
    A \emph{valuation} of $F$ is a surjective map $v \colon F \rightarrow \Gamma \cup \{\infty\}$, where the \emph{value group} $\Gamma$ is an abelian totally ordered group\footnote{
    An abelian totally ordered group $\Gamma$ is an abelian group equipped with a linear ordering $\leq$, such that $a \leq b \implies a + c \leq b + c$ for all $a, b, c \in \Gamma$.
    Here, the group law of $\Gamma$ is written additively.
    The ordering and the group law on $\Gamma$ can be extended to the set $\Gamma \cup \{\infty\}$ by defining $a \leq \infty$ and $a + \infty = \infty + a = \infty + \infty = \infty$ for all $a \in \Gamma$.
    }, such that the following conditions are satisfied for all $a, b \in F$:
    \begin{enumerate}[(i)]
        \item $v(a) = \infty$ if and only if $a = 0$,
        \item $v(ab) = v(a) + v(b)$,
        \item $v(a + b) \geq \min \{v(a), v(b)\}$, with equality if $v(a) \neq v(b)$.
    \end{enumerate}
    A valuation is called \emph{non-trivial} if $\Gamma \neq \{0\}$.
    A \emph{valued field} is a pair $(F, v)$ where $F$ is a field and $v$ is a valuation of $F$. 
    Its \emph{valuation ring} $A_v$ is defined as
    \[
        A_v \coloneqq \{a \in F \mid v(a) \geq 0\}.
    \]
    We have $A_v \neq F$ if and only if $v$ is non-trivial.
    $A_v$ is a ring with a unique maximal ideal
    \[
        M_v \coloneqq \{a \in F \mid v(a) > 0\}.
    \]  
The quotient $F_v \coloneqq A_v/M_v$ is called the \emph{residue field} of $(F, v)$.
It is indeed a field since $M_v$ is maximal.
A valuation $v$ is called a \emph{real place} of $F$ if the residue field $F_v$ is formally real.

Consider the field $F = \R(X)$.
The following proposition gives a well-known characterization (up to isomorphism of the value group $\Gamma$) of the set of all non-trivial real places $\R(X)$ whose valuation ring contains the subfield $\R$.
\begin{restatable}{prop}{propclass}\label{prop:class}
Let $v$ be a non-trivial real place of $\R(X)$ such that $\R \subseteq A_v$.
Then $v$ belongs to one of the two following types of real places:
\begin{enumerate}
    \item For every $t \in \R$ there is a real place $v_t \colon \R(X) \rightarrow \Z \cup \{\infty\}$, defined by $v_t(y) = a$, where $a \in \Z$ is such that $y$ can be written as $y = (X- t)^a \cdot \frac{f}{g}$, with $f, g$ being polynomials in $\R[X]$ not divisible by $X - t$.
    The residue field $\R(X)_{v_t}$ is isomorphic to $\R$ by the natural homomorphism $y + M_{v_t} \mapsto y(t)$.
    
    \item There is a real place $v_{\infty} \colon \R(X) \rightarrow \Z \cup \{\infty\}$, defined by $v_t(\frac{f}{g}) = \deg g - \deg f$, where $f, g$ are polynomials in $\R[X]$.
    The residue field $\R(X)_{v_{\infty}}$ is isomorphic to $\R$ by the natural homomorphism $y + M_{v_{\infty}} \mapsto \lim_{t \rightarrow \infty} y(t)$.
\end{enumerate}
\end{restatable}

Let $P$ be a semicone of a field $F$, and $F_0$ be a subfield of $F$.
Denote by $\leq$ the corresponding semiordering of $P$; define the set 
\begin{equation}\label{eq:defA}
    A_{F_0}^P \coloneqq \{a \in F \mid a \leq b \text{ and } -a \leq b \text{ for some } b \in F_0\}.
\end{equation}
The following lemmas show that $A_{F_0}^P$ is a valuation ring, and that its corresponding residue field admits a semiordering induced by $P$ under additional conditions.

\begin{lem}[{\cite[Lemma~7.13]{prestel2007lectures}}]\label{lem:713}
    Let $P$ be a semiordering of a field $F$ and $F_0$ a subfield of $F$.
    Then $A_{F_0}^P$ is a valuation ring of some valuation of $F$.
\end{lem}

\begin{lem}[{\cite[Lemma~7.15]{prestel2007lectures}}]\label{lem:715}
    Let $P$ be a semiordering of a field $F$ and $F_0$ a subfield of $F$, such that there exists $b \in F$ with $a \leq b$ for all $a \in F_0$.
    Let the valuation $v$ of $F$ correspond to $A_{F_0}^P$.
    Then $(A_v \cap P) / M_v$ is a semiordering of $F_v$.
\end{lem}

\section{Local-global principle over strictly positive polynomials}\label{sec:pospoly}
For a subset $B$ of $\R$, define the set $\W(B)$ of polynomials that are non-negative on $B$:
\[
    \W(B) \coloneqq \{f \in \R[X] \mid f(x) \geq 0 \text{ for all } x \in B\}.
\]
Obviously $\U(B) \subseteq \W(B)$.
For $f, g \in \W(\R) \setminus \{0\}$, by the fundamental theorem of algebra, one can write (uniquely)
\[
    f = c \prod_{j \in J} (x - r_j)^{d_j} \prod_{k \in K} (x^2 + a_k x + b_k)^{e_k}, \quad g = c' \prod_{j \in J} (x - r_j)^{d'_j} \prod_{k \in K} (x^2 + a_k x + b_k)^{e'_k}
\]
where $c, c', r_j, a_k, b_k \in \R$ and $d_j, d'_j, e_k, e'_k$ are non-negative integers, and the polynomials $x^2 + a_k x + b_k$ have no real root.
Here, $J$ indexes all real roots of $f$ and $g$, and $K$ indexes all conjugate pairs of imaginary roots of $f$ and $g$.
Since $f$ and $g$ are non-negative on $\R$, all $d_j$ and $d'_j$ are even, and $c, c'$ are positive.
Therefore, the \emph{greatest common divisor} of $f$ and $g$, defined by
\[
    \gcd(f, g) \coloneqq \prod_{j \in J} (x - r_j)^{\min\{d_j, d'_j\}} \prod_{k \in K} (x^2 + a_k x + b_k)^{\min\{e_k, e'_k\}}
\]
is also non-negative on $\R$.
It follows that the polynomials $\gcd(f, g)$, $f/\gcd(f, g)$ and $g/\gcd(f, g)$ are all in $\W(\R)$.

We now give a proof of our first main result, which can be considered as a ``strictly positive'' version of the Br\"ocker-Prestel local-global principle.
A comparison of our proof with the proof of the original theorem is given in Appendix~\ref{app:compare}.
\locglobpos*

\begin{proof}
The theorem is trivially true if $B$ is empty, hence we suppose $B \neq \emptyset$.
Suppose $f_1 h_1 + \cdots + f_n h_n = 0$ has no solution $(f_1, \ldots, f_n)$ over $\U(B)$.
Consider the following subset of the field $\R(X)$:
\[
    P_0 \coloneqq \left\{\frac{g}{G} \cdot \sum_{i = 1}^n f_i h_i, \text{ where all } f_i \in \U(B) \text{ and } g, G \in \W(\R) \setminus \{0\} \right\}.
\]
Since $f_1 h_1 + \cdots + f_n h_n = 0$ has no solution $(f_1, \ldots, f_n)$ over $\U(B)$, we have $0 \not\in P_0$.
We claim that $P'_0 = P_0 \cup \{0\}$ is a pre-semicone of $\R(X)$.
Indeed, we verify the three conditions given in Definition~\ref{def:presemicone}:
\begin{enumerate}[(i)]
    \item $P'_0 + P'_0 \subseteq P'_0$.
    It suffices to show $P_0 + P_0 \subseteq P_0$.
    Let $c = \frac{g}{G} \cdot \sum_{i = 1}^n f_i h_i$, $c' = \frac{g'}{G'} \cdot \sum_{i = 1}^n f'_i h_i$ be elements of $P_0$.
    Without loss of generality we can suppose $\gcd(g, G) = \gcd(g', G') = 1$.
    Write $d \coloneqq \gcd(g, g')$, $D \coloneqq \gcd(G, G')$, then the polynomials $d, \frac{g}{d}, \frac{g'}{d}, D, \frac{G}{D}, \frac{G'}{D}$ are all elements of $\W(\R) \setminus \{0\}$, and $\gcd(\frac{gG'}{dD}, \frac{g'G}{dD}) = 1$.
    Hence,
    \begin{equation}\label{eq:add}
        c + c' = \sum_{i = 1}^n \left(f_i \frac{g}{G} + f'_i \frac{g'}{G'}\right) h_i = \frac{dD}{GG'}\sum_{i = 1}^n \left( f_i \frac{gG'}{dD} + f'_i \frac{g'G}{dD} \right) h_i
    \end{equation}
    For any $x \in B$, we have $\frac{gG'}{dD} (x) \geq 0$ and $\frac{g'G}{dD} (x) \geq 0$.
    Since $\gcd(\frac{gG'}{dD}, \frac{g'G}{dD}) = 1$, the two polynomials $\frac{gG'}{dD}, \frac{g'G}{dD}$ cannot both vanish at $x$.
    Therefore either $\frac{gG'}{dD}(x) > 0$ or $\frac{g'G}{dD} (x) > 0$.
    Because $f_i(x) > 0$ and $f'_i(x) > 0$, it follows that $\left(f_i \frac{gG'}{dD} + f'_i \frac{g'G}{dD}\right) (x) > 0$.
    So $f_i \frac{gG'}{dD} + f'_i \frac{g'G}{dD} \in \U(B)$, and $c + c' \in P_0$.
    \item $\R(X)^2 \cdot P'_0 \subseteq P'_0$. This is obvious since $\R[X]^2 \cdot \W(\R) \subseteq \W(\R)$.
    \item $P'_0 \cap - P'_0 = \{0\}$.
    It suffices to show $P_0 \cap - P_0 = \emptyset$.
    On the contrary suppose $c \in P_0 \cap - P_0$, then $0 = c + (-c) \in P_0 + P_0 \subseteq P_0$, a contradiction.
\end{enumerate}
By Lemma~\ref{lem:113}, $P'_0$ can be extended to some $P$ such that either $P$ or $-P$ is a semicone of the field $\R(X)$.
Without loss of generality suppose $P \supseteq P'_0$ is a semicone, otherwise we can replace all $h_i$ by $- h_i$.
Since the field $\R(X)$ has no archimedean ordering~\cite[Example~1.1.4(2)]{prestel2013positive}, the \emph{semiordering} corresponding to $P$ must be non-archimedean (otherwise by Lemma~\ref{lem:120} it must be an archimedean ordering).
Consider the subfield $\R$ of $\R(X)$, by Lemma~\ref{lem:713} the valuation ring $A_{\R}^P$ (as defined in \eqref{eq:defA}) corresponds to some valuation $v$ of $\R(X)$.
Since $P$ is non-archimedean, there exists some $a \in \R(X)$ such that $a - r \in P$ for all $r \in \R$, hence $A_{\R}^P \neq \R(X)$. 
Also, Lemma~\ref{lem:715} shows that the residue field $F_v$ admits a semiordering $(P \cap A_v)/M_v$.
By Lemma~\ref{lem:115}, $F_v$ is formally real.
Therefore, $v$ is a non-trivial real place of $\R(X)$, and from the definition of $A_{\R}^P$ we have $\R \subseteq A_{\R}^P = A_v$.

Using the classification of real places of $\R(X)$ given in Proposition~\ref{prop:class}, consider the following three cases.
Since $F_v$ is isomorphic to $\R$, the semiordering $(P \cap A_v)/M_v$ corresponds to the only ordering on $\R$.
\begin{enumerate}
    \item The real place $v$ is equivalent to a place $v_t$ for some $t \in \oB \subseteq \R$.
    In this case $\R[X] \subseteq A_v$.
    We show that $h_i(t) \geq 0$ for all $i$.
    By symmetry it suffices to show $h_1(t) \geq 0$.
    For every $\varepsilon \in \Rpp$, we have $\varepsilon \in \U(B)$, so $h_1 + \varepsilon(h_2 + \cdots + h_n) \in P_0 \subseteq P$.
    Since $h_1 + \varepsilon(h_2 + \cdots + h_n) \in \R[X] \subseteq A_v$,
    we have $h_1 + \varepsilon(h_2 + \cdots + h_n) \in P \cap A_v$, which gives
    \begin{equation}\label{eq:modMv}
        h_1 + \varepsilon(h_2 + \cdots + h_n) + M_v \in (P \cap A_v)/M_v.
    \end{equation}
    Since the residue field $\R(X)_{v}$ is isomorphic to $\R$ by the natural homomorphism $y + M_{v} \mapsto y(t)$, Equation~\eqref{eq:modMv} yields
    \[
        h_1(t) + \varepsilon(h_2(t) + \cdots + h_n(t)) \geq 0.
    \]
    Since this is true for all $\varepsilon > 0$, we conclude that $h_1(t) \geq 0$ and thus $h_i(t) \geq 0$ for all $i$.
    
    \item The real place $v$ is equivalent to a place $v_t$ for some $t \in \R \setminus \oB$.
    There exists a polynomial $H_B \in \R[X]$, such that $H_B(x) > 0$ for all $x \in B$ but $H_B(t) < 0$.
    Indeed, since $t \not\in \oB$, there exists an interval $(t - \delta, t + \delta)$ disjoint from $B$; it then suffices to take $H_B \coloneqq (X - t)^2 - \delta^2$.
    
    As in the previous case, we have $h_1(t) \geq 0$.
    Furthermore, since $H_B \in \U(B)$ by its definition, we have $H_B h_1 + \varepsilon(h_2 + \cdots + h_n) \in P_0 \subseteq P$ for all $\varepsilon \in \Rpp$.
    This yields $(H_B h_1)(t) \geq 0$.
    However, we have $H_B(t) < 0$ by its definition.
    This together with $h_1(t) \geq 0$ yields $h_1(t) = 0$. 
    By symmetry we can prove $h_i(t) = 0$ for all $i$, this contradicts the condition $\gcd(h_1, \ldots, h_n) = 1$.

    \item The real place $v$ is equivalent to the place $v_{\infty}$.
    We divide $\{h_1, \ldots, h_n\}$ into two parts according to the parity of its degree.
    Without loss of generality, suppose $h_1, \ldots, h_k$ have even degree, and $h_{k+1}, \ldots, h_n$ have odd degree.
    
    Define the \emph{leading coefficient} of a polynomial as the coefficient of its highest degree monomial.
    First we claim that the leading coefficients of $h_1, \ldots, h_k$ are all positive.
    By symmetry, we only prove positivity of the leading coefficients of $h_1$.
    
    Let $m = \max\{\deg h_1, \ldots, \deg h_n\} + 1$.
    Since $(X^2+1)^m \in \U(B)$ and $X^{\deg h_1} \in \W(\R)$, we have 
    \begin{equation}\label{eq:modMvinf}
        \frac{h_1}{X^{\deg h_1}} + \frac{(X^2+1)^m}{(X^2+1)^{2m}}(h_2 + \cdots + h_n) + M_v \in (P \cap A_v)/M_v.
    \end{equation}
    Since the residue field $\R(X)_{v}$ is isomorphic to $\R$ by the natural homomorphism $y + M_{v_t} \mapsto \lim_{t \rightarrow \infty} y(t)$, Equation~\eqref{eq:modMvinf} shows that the leading coefficient of $h_1$ is positive. 
    Therefore by symmetry, the leading coefficient of $h_i$ is positive for all $1 \leq i \leq k$.
    
    We then separate four cases.
    \begin{enumerate}
        \item If $B$ is bounded, that is, $B \subset (a, b)$ for some $a, b \in \R$.
        Let $s > \max\{|a|, |b|\}$, then $X + s \in \U(B)$.
        Since $\deg h_n$ is odd, we have $X^{\deg h_n + 1} \in \W(\R)$.
        Therefore,
        \begin{equation}\label{eq:modMvinf1}
        \frac{(X+s) h_n}{X^{\deg h_n + 1}} + \frac{(X^2+1)^m}{(X^2+1)^{2m}}(h_1 + \cdots + h_{n-1}) + M_v \in (P \cap A_v)/M_v.
        \end{equation}
        This shows that the leading coefficient of $h_n$ is positive.
        
        However, we also have $- X + s \in \U(B)$, so we can 
        replace $(X+s)$ with $(- X + s)$ in Equation~\eqref{eq:modMvinf1}.
        This shows that the leading coefficient of $h_n$ is negative.
        Therefore $h_n$ does not exist, so all $h_1, \ldots, h_n$ must have even degree.
        But then $(X + 1 - a)(b + 1 - X) \in \U(B)$, so
        \begin{equation}\label{eq:modMvinf1bisbis}
        \frac{(X + 1 - a)(b + 1 - X) h_1}{X^{\deg h_1 + 2}} + \frac{(X^2+1)^m}{(X^2+1)^{2m}}(h_2 + \cdots + h_n) + M_v \in (P \cap A_v)/M_v.
        \end{equation}
        This shows that the leading coefficient of $h_1$ is negative, a contradiction.
        
        \item If $B \subset (a, +\infty)$ for some $a \in \R$, and $B$ contains arbitrary large positive reals, that is, $B \cap (b, +\infty) \neq \emptyset$ for all $b \in \R$.
        Then $X + 1 - a \in \U(B)$, so
        \begin{equation}\label{eq:modMvinf2}
        \frac{(X + 1 - a) h_n}{X^{\deg h_n + 1}} + \frac{(X^2+1)^m}{(X^2+1)^{2m}}(h_1 + \cdots + h_{n-1}) + M_v \in (P \cap A_v)/M_v.
        \end{equation}
        This shows that the leading coefficient of $h_n$ is positive.
        By symmetry, the leading coefficients of $h_{k+1}, \ldots, h_n$ are all positive.
        Therefore, for large enough $t \in B$, $h_1(t), \ldots, h_n(t)$ are all positive.
        
        \item If $B \subset (-\infty, a)$ for some $a \in \R$, and $B$ contains arbitrary small reals, that is, $B \cap (-\infty, b) \neq \emptyset$ for all $b \in \R$.
        Then $a + 1 - X \in \U(B)$, so
        \begin{equation}\label{eq:modMvinf3}
        \frac{(a + 1 - X) h_n}{X^{\deg h_n + 1}} + \frac{(X^2+1)^m}{(X^2+1)^{2m}}(h_1 + \cdots + h_{n-1}) + M_v \in (P \cap A_v)/M_v.
        \end{equation}
        This shows that the leading coefficient of $h_n$ is negative.
        By symmetry, the leading coefficients of $h_{k+1}, \ldots, h_n$ are all negative.
        Therefore, for small enough $0 > t \in B$, $h_1(t), \ldots, h_n(t)$ are all positive.
        
        \item If $B$ contains arbitrary large and arbitrary small reals.
        We claim that the leading coefficients of $h_{k+1}, \ldots, h_n$ all have the same sign.
        Suppose on the contrary that they have different signs, denote by $a_i$ the leading coefficient of $h_i$, so $h_i = a_i X^{\deg h_i} + H_i$ for some polynomial $H_i$ of degree at most $\deg h_i - 1$.
        Then there exist \emph{strictly positive} reals $r_{k+1}, \ldots, r_n$ such that $r_{k+1} a_{k+1} + \cdots + r_n a_{n} = 0$.
        Then, for any $s \in \R$, we have $X^2 - 2sX + s^2 + 1 \in \U(B)$, so
        \begin{multline}\label{eq:modMvinf4}
         \frac{(X^2+1)^m}{(X^2+1)^{2m}}(h_1 + \cdots + h_{k}) + \frac{r_{k+1}(X^2 - 2sX + s^2 + 1)}{X^{\deg h_{k+1} + 1}} h_{k+1} \\
         + \frac{r_{k+2}}{X^{\deg h_{k+2} - 1}} h_{k+2} + \cdots + \frac{r_{n}}{X^{\deg h_{n} - 1}} h_{n} + M_v \in (P \cap A_v)/M_v.
        \end{multline}
        The limit of the left hand side when $X$ tends to infinity is equal to
        \begin{align*}
        g(s) & \coloneqq \lim_{X \rightarrow \infty}\left(\frac{r_{k+1}(X^2 - 2sX)}{X^{\deg h_{k+1} + 1}} h_{k+1}(X) + \sum_{j = k+2}^n \frac{r_{j}}{X^{\deg h_{j} - 1}} h_{j}(X)\right) \\
        & = -2s a_{k+1} r_{k+1} + \lim_{X \rightarrow \infty}\left(\sum_{j = k+1}^n \frac{r_{j}}{X^{\deg h_{j} - 1}} H_{j}(X)\right)
        \end{align*}
        because $r_{k+1} a_{k+1} + \cdots + r_n a_{n} = 0$.
        According to whether $a_{k+1} r_{k+1}$ is positive or negative, we can take a positive or negative $s$ with large enough absolute value, so that the value of $g(s)$ is negative.
        This contradicts Equation~\eqref{eq:modMvinf4}, which shows that the limit of the left hand side when $X \rightarrow \infty$ is positive.
        
        We therefore conclude that the leading coefficients of $h_{k+1}, \ldots, h_n$ all have the same sign.
        If they are positive, then for large enough $t \in B$, $h_1(t), \ldots, h_n(t)$ are all positive.
        If they are negative, then for small enough $t \in B$, $h_1(t), \ldots, h_n(t)$ are all positive.
    \end{enumerate}
\end{enumerate}
To sum up, in all possible cases, we have $t \in \oB$ with $h_i(t) \geq 0$ for all $i$.
If $-P$ is a semicone instead of $P$, analogously we can find $t \in \oB$ such that $h_i(t) \leq 0$ for all $i$.
\end{proof}

\section{Local-global principle over $\N[X]$}\label{sec:locglob}
In this section we prove Theorem~\ref{thm:mainlocglob}.
Omitted proofs are given in Appendix~\ref{app:proofs}.
The key to bridging the difference between the semirings $\U(B)$ and $\N[X]$ is P\'{o}lya's Theorem:

\begin{lem}[P\'{o}lya's Theorem~{\cite[Theorem~56]{hardy1952g}}]\label{lem:polya}
    If a homogeneous polynomial $f \in \R[X_1, \ldots, X_n]$ is strictly positive for all $(X_1, \ldots, X_n)$ on $\left( \Rp \right)^n \setminus \{0\}$,
    then there exists $p \in \N$ such that
    $
        (X_1 + \cdots + X_n)^p \cdot f \in \Rp[X_1, \ldots, X_n].
    $
\end{lem}

The following proposition reduces Theorem~\ref{thm:mainlocglob} to real polynomials.
\begin{restatable}{prop}{propNtoR}\label{prop:NtoR}
    Given $h_1, \ldots, h_n \in \Z[X]$.
    The equation $f_1 h_1 + \cdots + f_n h_n = 0$ has a solution $(f_1, \ldots, f_n)$ over $\N[X] \setminus \{0\}$ if and only if it has a solution over $\Rp[X] \setminus \{0\}$.
\end{restatable}

The next proposition further reduces it to $\U(\Rpp)$.
The key to its proof is Lemma~\ref{lem:polya}.
\begin{restatable}{prop}{propRtoU}\label{prop:RtoU}
    Given $h_1, \ldots, h_n \in \Z[X]$.
    The equation $f_1 h_1 + \cdots + f_n h_n = 0$ has a solution $(f_1, \ldots, f_n)$ over $\Rp[X] \setminus \{0\}$ if and only if it has a solution over $\U(\Rpp)$.
\end{restatable}

This justifies the need for a ``strictly positive'' version of the Br\"ocker-Prestel principle, since Proposition~\ref{prop:RtoU} no longer holds if we replace $\U(\Rpp)$ with $\W(\Rpp) \setminus \{0\}$ (see Remark~\ref{rmk:UvsW}).

We now prove the local-global principle for homogeneous linear equations over $\N[X]$.
\mainlocglob*
\begin{proof}
    Suppose the equation $f_1 h_1 + \cdots + f_n h_n = 0$ has no solution $(f_1, \ldots, f_n)$ over $\N[X] \setminus \{0\}$.
    By Proposition~\ref{prop:NtoR} and \ref{prop:RtoU}, it has no solution over $\U(\Rpp)$.
    Hence, by Theorem~\ref{thm:locglobpos}, there exists a real number $t \in \overline{\Rpp} = \Rp$ such that $h_i(t)$ are all non-negative or all non-positive.
\end{proof}

\section{Decidability}\label{sec:dec}
In this section we show our main decidability result.
\decidability*
\begin{proof}
    (A summary of the algorithm constructed in this proof is given in Appendix~\ref{app:alg}.)
    
    By the homogeneity of the linear equation, we can divide $h_1, \ldots, h_n$ by their greatest common divisor and suppose $\gcd(h_1, \ldots, h_n) = 1$.
    Computing the greatest common divisor can be done in polynomial time using the Euclidean algorithm.
    
    We then show that we can simplify the equation so that $h_1, \ldots, h_n$ satisfy
    \begin{equation}\label{eq:condpm}
        h_i(0) > 0, h_j(0) < 0, \quad \text{ for some } i, j.
    \end{equation}
    Suppose this is not already the case, that $h_i(0) \geq 0$ for all $i$ or $h_i(0) \leq 0$ for all $i$.
    Without loss of generality suppose $h_i(0) \geq 0$ for all $i$.
    We write $h_1(0) = 0, \ldots, h_k(0) = 0, h_{k+1}(0) > 0, \ldots, h_n(0) > 0$.
    Then $X \mid h_i$ for $i = 1, \ldots, k$.
    
    If $k = 0$, that is $h_i(0) > 0$ for all $i$, then $f_1 h_1 + \cdots + f_n h_n = 0$ has no solution over $\N[X] \setminus \{0\}$.
    Indeed, suppose on the contrary that $(f_1, \ldots, f_n)$ is such a solution.
    Dividing all $f_i$ by a suitable power of $X$ we can suppose $f_s(0) \neq 0$ for some $s$.
    Then $f_i(0) \geq 0$ for all $i$ while $f_s(0) > 0$, which yields $f_1(0) h_1(0) + \cdots + f_n(0) h_n(0) > 0$, a contradiction.

    If $k \geq 1$, we show that the equation
    \begin{equation}\label{eq:lin}
        f_1 h_1 + \cdots + f_n h_n = 0
    \end{equation}
    has a solution over $\N[X] \setminus \{0\}$ if and only if the equation
    \begin{equation}\label{eq:linred}
        f_1 \cdot \frac{h_1}{X} + \cdots + f_{k} \cdot \frac{h_k}{X} + f_{k+1} h_{k+1} + \cdots + f_n h_n = 0
    \end{equation}
    has a solution over $\N[X] \setminus \{0\}$.
    Let $(f_1, \ldots, f_n)$ be a solution over $\N[X] \setminus \{0\}$ of Equation~\eqref{eq:lin}, then $f_1(0) h_1(0) + \cdots + f_n(0) h_n(0) = 0$.
    Since $h_i(0) = 0$ for all $i = 1, \ldots, k$, $h_i(0) > 0$ for $i = k+1, \ldots, n$ and $f_i(0) \geq 0$ for $i = 1, \ldots, n$, we must have $f_{k+1}(0) = 0, \ldots, f_n(0) = 0$.
    That is, $X \mid f_{k+1}, \ldots, X \mid f_{n}$.
    Therefore $(f_1, \ldots, f_k, f_{k+1}/X, \ldots, f_n/X)$ is a solution over $\N[X] \setminus \{0\}$ of Equation~\eqref{eq:linred}.
    This shows that we can divide $h_1, \ldots, h_k$ by $X$ without changing the existence of solutions of Equation~\eqref{eq:lin}.
    Repeating this division process, one eventually terminates by obtaining $h_i$ such that either: $h_i(0)$ are all strictly positive or all strictly negative, in which case Equation~\eqref{eq:lin} has no solution over $\N[X] \setminus \{0\}$; or $h_i(0) > 0$ and $h_j(0) < 0$ for some $i, j$, in which case we have achieved the desired simplification to Condition~\eqref{eq:condpm}.
    This procedure is repeated at most $\deg h_1 + \cdots + \deg h_n$ times, and therefore terminates in polynomial time.

    Supposing Condition~\eqref{eq:condpm}, we claim that $f_1 h_1 + \cdots + f_n h_n = 0$ has no solution over $\N[X] \setminus \{0\}$ if and only if there exists $t \in \Rp$ such that $h_i(t)$ are all non-positive or all non-negative.
    The first implication is given by Theorem~\ref{thm:mainlocglob}.
    Conversely, suppose $h_i(t)$ are all non-positive or all non-negative. 
    Without loss of generality suppose $h_i(t) \geq 0$ for all $i$.
    By Condition~\eqref{eq:condpm}, we have $t \neq 0$.
    Suppose on the contrary that $(f_1, \ldots, f_n)$ is a solution over $\N[X] \setminus \{0\}$, then $f_i(t) > 0$ for all $i$ since $t > 0$.
    Since $\gcd(h_1, \ldots, h_n) = 1$, at least one of $h_i(t)$ must be non-zero.
    Since $h_i(t) \geq 0$ for all $i$, we have $f_1(t) h_1(t) + \cdots + f_n(t) h_n(t) > 0$, a contradiction.

    Thus, it suffices to decide whether there exists $t \geq 0$ such that $h_i(t)$ are all non-positive or all non-negative.
    This can be expressed in the existential theory of the reals:
    \begin{equation}\label{eq:sentreal}
        \exists X \, \left(X \geq 0 \land h_1(X) \geq 0 \land \cdots \land h_n(X) \geq 0 \right) \lor \left(X \geq 0 \land h_1(X) \leq 0 \land \cdots \land h_n(X) \leq 0 \right).
    \end{equation}
    Deciding the existential theory of the reals \emph{in one variable} can be done in polynomial time with respect to the total bit length used to encode the sentence, due to a classic result by Collins\footnote{The algorithm by Collins~\cite{collins1975quantifier} has complexity $L^3(nd)^{2^{O(K)}}$, where $L$ is the total coefficient bit length, $n$ the number of polynomials, $d$ the total degree of the polynomials, and $K$ the number of variables. In the one variable case, $K=1$, the algorithm takes polynomial time with respect to the total bit length.}~\cite{collins1975quantifier}.
    Therefore, one can decide the correctness of the sentence~\eqref{eq:sentreal} in polynomial time.
    Combining all the steps, we conclude that the total complexity is in PTIME.
\end{proof}

\section{Application to wreath product}\label{sec:appwr}
In this section we show the following result on wreath products.
\wreathdec*

Let $\varphi$ be the isomorphism defined in~\eqref{eq:defphi}.
Fix a finite set of elements $\mG$ as in Theorem~\ref{thm:wreathdec}.
For $i = 1, \ldots, n$, denote by $H_i \in \Z[X, X^{-1}]$ the Laurent polynomial in the upper-right entry of the image of $\varphi((y_i, b_i))$.
Write $\mG = \mG_+ \cup \mG_-$ where $\mG_+ \coloneqq \{(y_i, b_i) \in \mG \mid b_i = 1\}$ and $\mG_- \coloneqq \{(y_j, b_j) \in \mG \mid b_j = -1\}$.
Let $\varphi(\mG), \varphi(\mG_+), \varphi(\mG_-)$ be the set of matrices that are images under $\varphi$ of elements in $\mG, \mG_+, \mG_-$.
Define the sets of indices 
\[
    I \coloneqq \{i \mid b_i = 1\}, \quad J \coloneqq \{i \mid b_i = -1\}.
\]
For simplicity, we write $A_i, i \in I$ for the matrices in $\varphi(\mG_+)$, and $B_j, j \in J$ the matrices in $\varphi(\mG_-)$.
For every tuple $(i, j) \in I \times J$, define the Laurent polynomial
\begin{equation}\label{eq:defhij}
    h_{ij} \coloneqq X^{-1} H_i + H_j \in \Z[X, X^{-1}].
\end{equation}
This is the upper-right entry of the matrix $A_i B_j$.

For a subset $S \subseteq I \times J$, denote by $\pi_I(S)$ its projection onto the $I$ coordinates, that is, $\pi_I(S) \coloneqq \{i \in I \mid \exists j \in J, (i, j) \in S\}$.
Define $\pi_J(S)$ likewise.
The key to proving the partial decidability of the Group Problem in $\Z \wr \Z$ is the following proposition that relates sub-semigroups of $\Z \wr \Z$ to equations over $\N[X] \setminus \{0\}$.

\begin{prop}\label{prop:wrtopos}
    Given a set $\mG = \mG_+ \cup \mG_-$ of generators defined as above.
    Let $h_{ij} \in \Z[X, X^{-1}]$ be the polynomials defined in \eqref{eq:defhij}.
    The semigroup $\langle \mG \rangle$ is a group if and only if there exists a set $S \subseteq I \times J$ satisfying $\pi_I(S) = I, \pi_J(S) = J$, such that the equation $\sum_{(i, j) \in S} f_{ij} h_{ij} = 0$ has a solution $(f_{ij})_{(i, j) \in S}$ over $\N[X] \setminus \{0\}$.
\end{prop}

\begin{proof}
    For a word $w$ in the alphabet $\varphi(\mG)$, define its \emph{product} $\pi(w)$ to be the matrix obtained by multiplying all the matrices in $w$ consecutively.
    Denote by $|w|_{+}$ (respectively, $|w|_{-}$) the number of letters in $w$ belonging in $\varphi(\mG_+)$ (respectively, $\varphi(\mG_-)$).
    Define the \emph{height} of the word $w$ to be $h(w) \coloneqq |w|_{+} - |w|_{-}$, then we have
    $
    \pi(w) = 
    \begin{pmatrix}
            1 & * \\
            0 & X^{h(w)}
    \end{pmatrix},
    $
    where $*$ is some element in $\Z[X, X^{-1}]$.

    For a finite alphabet $\mathcal{A}$, denote by $\mathcal{A}^+$ the set of non-empty words over $\mathcal{A}$.
    We claim that for any non-empty word $w \in \varphi(\mG)^+$ such that $h(w) = 0$, the upper right entry of $\pi(w)$ can be written as a sum $\sum_{(i,j) \in I \times J} f_{ij} h_{ij}$, where $f_{ij}$ are elements in $\N[X, X^{-1}]$.
    We prove this by induction the length of the word $w$.
    For the sake of simplicity, denote $U(\pi(w))$ the upper right entry of $\pi(w)$.
    
    If $w$ has length at most two, then it must be of the form $A_i B_j$ or $B_j A_i$, and the claim is easy to verify.
    Suppose the claim is true for all words $w$ of length less then $\ell > 2$.
    We prove the claim for words $w$ of length $\ell$.
    Distinguish the following two cases.
    
    \begin{enumerate}[1.]
    
        \item \textbf{The word $w$ is of the form $A_i w' B_j$ or $B_j w' A_i$ for some $i \in I, j \in J, w' \in \varphi(\mG)^+$.}
        Since $w'$ has length at most $\ell-2$ and is of height $0$, by induction hypothesis, 
        $\pi(w') = \begin{pmatrix}
            1 & r \\
            0 & 1
        \end{pmatrix}$, 
        with $r$ a linear combination of $h_{ij}$ with coefficients in $\N[X, X^{-1}]$. 
        If $w = A_i w' B_j$, then 
        \[
        \pi(w) = 
        \begin{pmatrix}
            1 & H_i \\
            0 & X
        \end{pmatrix}    
        \begin{pmatrix}
            1 & r \\
            0 & 1
        \end{pmatrix}
        \begin{pmatrix}
            1 & H_j \\
            0 & X^{-1}
        \end{pmatrix}
        =
        \begin{pmatrix}
            1 & X^{-1} r + (X^{-1} H_i + H_j) \\
            0 & 1
        \end{pmatrix}    
        =
        \begin{pmatrix}
            1 & X^{-1} r + h_{ij} \\
            0 & 1
        \end{pmatrix}.    
        \]
        So $U(\pi(w)) = X^{-1} r + h_{ij}$ can also be written as a linear combination of $h_{ij}, i \in I, j \in J$ with coefficients in $\N[X, X^{-1}]$.
        If $w = B_j w' A_i$, then 
        \[
        \pi(w) = 
        \begin{pmatrix}
            1 & H_j \\
            0 & X^{-1}
        \end{pmatrix}  
        \begin{pmatrix}
            1 & r \\
            0 & 1
        \end{pmatrix}
        \begin{pmatrix}
            1 & H_i \\
            0 & X
        \end{pmatrix} 
        =
        \begin{pmatrix}
            1 & X r + H_i + X H_j \\
            0 & 1
        \end{pmatrix}    
        =
        \begin{pmatrix}
            1 & X(r + h_{ij}) \\
            0 & 1
        \end{pmatrix}. 
        \]
        So $U(\pi(w)) = X(r + h_{ij})$ can also be written as a linear combination of $h_{ij}, i \in I, j \in J$ with coefficients in $\N[X, X^{-1}]$.
        
        \item \textbf{The word $w$ is of the form $A_i w' A_{i'}$ or $B_{j} w' B_{j'}$ for some $i, i' \in I$ or $j, j' \in J$.}
        First suppose $w = A_i w' A_{i'}$.
        Since $h(A_i) = 1 > 0$ and $h(A_i w') = -1 < 0$, there must exist a strict prefix $v$ of $w$ with height zero.
        This is because by reading the word $w$ letter by letter, this height of consecutive prefixes differs by at most one.
        We have $w = v v'$ with $h(v) = h(v') = 0$ where $v, v'$ are non-empty words.
        By induction hypothesis, $U(\pi(v)), U(\pi(v'))$ can be written as a linear combination of $h_{ij}$ with coefficients in $\N[X, X^{-1}]$.
        Therefore $U(\pi(w)) = U(\pi(v)) + U(\pi(v'))$ also satisfies this claim.
        The case where $w = B_{j} w' B_{j'}$ is completely analogous.
    \end{enumerate}
    Combining the two cases concludes the induction.
    It is easy to see from the induction process that if the letter $A_i$ appears in $w$, then the coefficient of the term $h_{ij}$ in the linear combination is not zero for some $j \in J$. 
    This is because at some point we have replaced $r$ with either $X^{-1} r + h_{ij}$ or $X(r + h_{ij})$.
    Similarly, if the letter $B_j$ appears in $w$, then the coefficient of the term $h_{ij}$ in the linear combination is non-zero for some $i \in I$.

    If the semigroup $\langle \mG \rangle$ is a group, then there exists a word $v$ in the alphabet $\mG$ using all letters in $\mG$, whose corresponding product is the neutral element.
    Taking the image under $\varphi$ yields a word $w = \varphi(v)$ in the alphabet $\varphi(\mG)$ such that $h(w) = 0$ and $U(\pi(w)) = 0$.
    The claim above and the discussion following it show that there exist Laurent polynomials $f_{ij} \in \N[X, X^{-1}]$ such that $\sum_{(i, j) \in I \times J} f_{ij} h_{ij} = 0$.
    Furthermore, all letters $A_i, i \in I$ and $B_j, j \in J$ appear in $w$, so for every $i$, the coefficient $f_{ij}$ in the linear combination is not zero for some $j \in J$; and for every $j$, the coefficient $f_{ij}$ is not zero for some $i \in I$.
    Let $S \coloneqq \{(i, j) \in I \times J \mid f_{ij} \neq 0\}$, then $\sum_{(i, j) \in S} f_{ij} h_{ij} = 0$, and $\pi_I(S) = I, \pi_J(S) = J$.
    By the homogeneity of the equation $\sum_{(i, j) \in S} f_{ij} h_{ij} = 0$, one can multiply all $f_{ij}$ by the monomial $X^n$ for a sufficiently large $n$, and suppose $f_{ij} \in \N[X] \setminus \{0\}$ instead of $\N[X, X^{-1}] \setminus \{0\}$. 
    This completes the proof of the first direction of implication in Proposition~\ref{prop:wrtopos}.

    For the other direction of implication, suppose there exists a set $S \subseteq I \times J$ satisfying $\pi_I(S) = I, \pi_J(S) = J$, such that the equation $\sum_{(i, j) \in S} f_{ij} h_{ij} = 0$ has a solution $(f_{ij})_{(i, j) \in S}$ over $\N[X] \setminus \{0\}$.
    By the homogeneity of the equation, suppose that there is a tuple $(u,v) \in S$ such that $X \nmid f_{uv}$.
    Let $(y,z) \in S$ be a tuple such that $\deg f_{yz} \geq \deg f_{ij}$ for all $(i,j) \in S$.

    Denote by $\Np[X]$ the set of polynomials of the form $\sum_{i = 0}^d a_i X^i$, where $d \geq 0$ and $a_i > 0$ for all $i$.
    By multiplying all $f_{ij}$ by the polynomial $(1 + X)^m$ for a sufficiently large $m$, we can suppose that $f_{uv} \in \Np[X]$, $X^{-v_0(f_{yz})} f_{yz} \in \Np[X]$, and $\deg f_{uv} \geq v_0(f_{yz})$.
    Indeed, we can take any $m \geq \max\{\deg f_{uv}, \deg X^{-v_0(f_{yz})} f_{yz}, v_0(f_{yz})\}$.
    Additionally, the condition that $\deg f_{yz} \geq \deg f_{ij}$ for all $(i,j) \in S$ is still satisfied after this multiplication.

    We now construct a word $w \in \varphi(\mG)^+$ that uses every letter in $\varphi(\mG)$, such that $h(\pi(w)) = 0$, $U(\pi(w)) = \sum_{(i, j) \in S} f_{ij} h_{ij} = 0$.
    We start with the word 
    \[
        w_0 \coloneqq A_u^{\deg f_{uv}} A_y^{\deg f_{yz} - \deg f_{uv}} B_z^{\deg f_{yz} - \deg f_{uv}} B_v^{\deg f_{uv}},
    \]
    which has height 0, and whose product has upper-right entry 
    \[
        U(\pi(w_0)) = h_{uv} \cdot \sum_{i = 0}^{\deg f_{uv}-1} X^i + h_{yz} \cdot \sum_{i = \deg f_{uv}}^{\deg f_{yz}-1} X^i.
    \]
    Since $f_{uv} \in \Np[X]$, $X^{-v_0(f_{yz})} f_{yz} \in \Np[X]$, and $\deg f_{uv} \geq v_0(f_{yz})$, the polynomials $\hat{f}_{uv} \coloneqq f_{uv} - \sum_{i = 0}^{\deg f_{uv}-1} X^i$, $\hat{f}_{yz} \coloneqq f_{yz} - \sum_{i = \deg f_{uv}}^{\deg f_{yz}-1} X^i$ are still polynomials in $\N[X] \setminus \{0\}$.
    For $(i, j) \in S$, define 
    \begin{align*}
    \hat{f}_{ij} \coloneqq
        \begin{cases}
            \hat{f}_{uv} & (i, j) = (u, v) \\
            \hat{f}_{yz} & (i,j) = (y, z) \\
            f_{ij} & \text{otherwise}.
        \end{cases}
    \end{align*}
    These are elements in $\N[X] \setminus \{0\}$ and satisfy
    $
        U(\pi(w_0)) + \sum_{(i,j) \in S} \hat{f}_{ij} h_{ij} = \sum_{(i, j) \in S} f_{ij} h_{ij} = 0.
    $
    
    We then gradually insert ``loops'' of the form $A_i B_j$ into the word $w_0$. This insertion does not change the height of the word, but it adds a multiple of $h_{ij}$ to the upper-right entry of the product.
    Indeed, if $h(vv') = 0$, then we have $h(v A_i B_j v') = 0$ and $U(\pi(v A_i B_j v')) = U(\pi(v v')) + X^{h(v)} h_{ij}$.
    Note that the initial word $w_0$ has suffixes of all heights from $0$ to $\deg f_{yz}$.
    For each $k = 0, \ldots, \deg f_{yz}$ and each $(i, j) \in S$, after a suffix of height $k$, we insert $\Coef_{X^k} (\hat{f}_{ij})$ times the ``loop'' $A_i B_j$, where $\Coef_{X^k} (\hat{f}_{ij})$ is the coefficient of the monomial $X^k$ in the polynomial $\hat{f}_{ij}$.
    The upper-right entry of the product after all these insertions will be
    \[
        U(\pi(w_0)) + \sum_{k = 0}^{\deg f_{yz}} \sum_{(i,j) \in S} \Coef_{X^k} (\hat{f}_{ij}) X^k \cdot h_{ij} = U(\pi(w_0)) + \sum_{(i,j) \in S} \hat{f}_{ij} h_{ij} = 0,
    \]
    because $\deg f_{yz} \geq \deg f_{ij}$ for all $(i,j) \in I \times J$.
    See Figure~\ref{fig:graph} for an example of this construction.
    
    We have thus constructed a word $w \in \varphi(\mG)^+$ such that $h(\pi(w)) = 0$, $U(\pi(w)) = 0$.
    Note that we have inserted at least one loop $A_i B_j$ for each $(i, j) \in S$.
    Since $\pi_I(S) = I, \pi_J(S) = J$, the word $w$ contains every letter $A_i, i \in I$ and $B_j, j \in J$.
    Because $\pi(w)$ is the neutral element, the inverse of every letter in $w$ can be written as a product of matrices in $\varphi(\mG)$.
    Indeed, if $w = v X v'$ then $X^{-1} = \pi(v' v)$.
    Thus the inverse of every element of $\varphi(\mG)$ is in $\langle\varphi(\mG)\rangle$.
    We conclude that $\langle \varphi(\mG) \rangle$, and thus $\langle \mG \rangle$, is a group.
\end{proof}

\begin{figure}[ht]
\centering
\includegraphics[width=\textwidth]{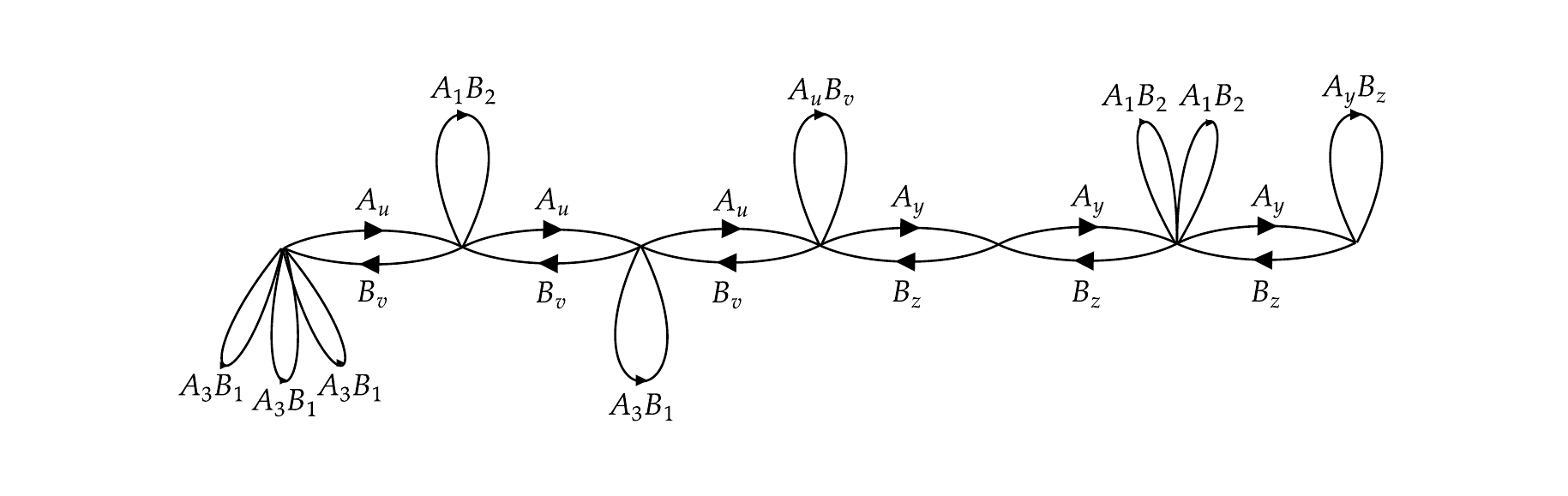}
\caption{Example of a word constructed in the proof of Proposition~\ref{prop:wrtopos}. 
Here, $S = \{(u, v), (y, z), (1, 2), (3, 1)\}$, and $f_{uv} = 1 + X + X^2 + X^3$, $f_{yz} = X^3 + X^4 + X^5 + X^6$, $f_{12} = X + 2 X^5$, $f_{31} = 3 + X^2$.
The constructed word is $A_u (A_1 B_2) A_u A_u (A_u B_v) A_y A_y (A_1 B_2)(A_1 B_2) A_y (A_y B_z) B_z B_z B_z B_z (A_3 B_1) B_z B_z (A_3 B_1)(A_3 B_1)(A_3 B_1)$.}
\label{fig:graph}
\end{figure}

We have thus established the link between the Group Problem in $\Z \wr \Z$ and homogeneous linear equations over $\N[X]$.
Theorem~\ref{thm:wreathdec} follows from Proposition~\ref{prop:wrtopos} and the decidability result of Theorem~\ref{thm:decidability}.
Its proof is given in Appendix~\ref{app:proofs}.


\newpage

\bibliography{pospoly}

\appendix

\section{Omited proofs}\label{app:proofs}
\propclass*
\begin{proof}
Since $\R \subseteq A_v$, every element $r \in \R \setminus \{0\}$ satisfies $v(r) \geq 0$ and $v(r^{-1}) \geq 0$.
But $v(r) + v(r^{-1}) = v(1) = 0$, so $v(r) = 0$.
Consider the value $v(X)$, there are two possibilities:
\begin{enumerate}
    \item If $v(X) \geq 0$.
    In this case, we have $\R \subseteq A_v$ and $X \in A_v$, therefore $\R[X] \subseteq A_v$.
    Since $M_v$ is a maximal (hence prime) ideal of $A_v$, the ideal $\R[X] \cap M_v$ is a prime ideal of $\R[X]$. 
    Furthermore, $\R[X] \cap M_v$ is not zero, otherwise every element of $\R[X] \setminus \{0\}$ would be invertible in $A_v$, so $\R(X) \subseteq A_v$, contradicting the non-triviality of $v$.
    Since $\R[X]$ is a principle ideal domain, the non-zero prime ideal $\R[X] \cap M_v$ is generated by a single irreducible polynomial in $\R[X]$.
    Consider the two cases:
    \begin{enumerate}
        \item The ideal $\R[X] \cap M_v$ is generated by a polynomial $X - t$ for some $t \in \R$.
        In this case we have $v(X - t) > 0$.
        Every polynomial $f \in \R[X]$ not divisible by $(X - t)$ can be written as $f = (X - t) \cdot F + r$ for some $F \in \R[X]$, $r \in \R \setminus \{0\}$.
        Since $v((X - t) \cdot F) = v(X - t) + v(F) > 0$ and $v(r) = 0$, we have $v(f) = v(r) = 0$.
        
        Every element $y \in \R(X)$ can be written as $y = (X- t)^a \cdot \frac{f}{g}$, where $f, g$ are polynomials in $\R[X]$ not divisible by $(X - t)$.
        Then $v(y) = a \cdot v(X- t) + v(f) - v(g) = a v(X - t)$.
        Under isomorphism of the value group $\Gamma$, we can without loss of generality we can suppose $v(X - t) = 1$, then we get the valuation $v_t$ of type 1 described in the proposition.
        Since every element $y \in M_{v_t}$ satisfies $y(t) = 0$, we have that $y + M_{v_t} \mapsto y(t)$ is an isomorphism from the residue field to $\R$; it is a formally real field.
        \item The ideal $\R[X] \cap M_v$ is generated by a polynomial $X^2 + c X + d$ without real roots.
        In this case, the residue field $A_v / M_v$ is a quadratic extension of $\R$, and is hence isomorphic to the field $\C$.
        However $\C$ is not formally real.
        Indeed, suppose on the contrary that $\C$ admits some ordering $\leq$, then since $0 < i^2 = -1$ and $0 < 1^2 = 1$, we have $0 < (-1) + 1 = 0$, a contradiction.
    \end{enumerate}
    \item If $v(X) < 0$. 
    In this case we have $\R[1/X] \subseteq A_v$ and $1/X \in M_v$.
    Since $\R[1/X] \cap M_v$ is a prime ideal of $\R[1/X]$ that contains $1/X$, it is generated by $1/X$.
    Then similar to the case 1.a., every element $y \in \R(X)$ can be written as $y = (1/X)^a \cdot \frac{F}{G}$, where $F, G$ are polynomials in $\R[1/X]$ not divisible by $1/X$.
    Without loss of generality suppose $v(1/X) = 1$, we have $v(y) = a$. Rewrite $y = \frac{f}{g}$, comparing degrees, we have $a = \deg g - \deg f$.
    So $v$ is the valuation $v_{\infty}$ of type 2 described in the proposition.
    Since every element $y \in M_{v_{\infty}}$ satisfies $\lim_{t \rightarrow \infty} y (t) = 0$, we have that $y + M_{v_{\infty}} \mapsto \lim_{t \rightarrow \infty} y(t)$ is an isomorphism from the residue field to $\R$; it is a formally real field.
\end{enumerate}
\end{proof}

\propNtoR*
\begin{proof}
A solution over $\N[X] \setminus \{0\}$ is obviously also a solution over $\Rp[X] \setminus \{0\}$.
Conversely, let $f_i = \sum_{j = 0}^{d_i} a_{ij} X^j, i = 1 \ldots, n,$ be a solution of $f_1 h_1 + \cdots + f_n h_n = 0$.
Write $h_i = \sum_{j = 0}^{e_i} b_{ij} X^j, i = 1 \ldots, n$, then the equation $f_1 h_1 + \cdots + f_n h_n = 0$ is equivalent to the system of equations
\begin{equation}\label{eq:syslin}
    \sum_{i = 1}^n \sum_{j = 0}^d a_{ij} b_{i, d - j} = 0, \quad d = 1, \ldots, \max_{1 \leq i \leq n}{(d_i + e_i)}.
\end{equation}
All the coefficients $b_{ij}$ are integers, and $b_{i, d - j} = 0$ whenever $d - j < 0$.

If $f_1 h_1 + \cdots + f_n h_n = 0$ has a solution over $\Rp[X] \setminus \{0\}$, then System~\eqref{eq:syslin} has a solution $a_{ij}, i = 1, \ldots, n, j = 1, \ldots, d_i$ over $\R$, satisfying 
\begin{equation}\label{eq:syspos}
    a_{ij} \geq 0, \quad i = 1, \ldots, n, \quad j = 1, \ldots, d_i,
\end{equation}
and
\begin{equation}\label{eq:sysineq}
    a_{i1} \neq 0 \text{ or } a_{i2} \neq 0 \text{ or } \ldots \text{ or } a_{i d_i} \neq 0, \quad i = 1, \ldots, n.
\end{equation}
This condition is a boolean combination of homogeneous linear inequalities with integer coefficients.
Since the linear Systems~\eqref{eq:syslin}, \eqref{eq:syspos} and \eqref{eq:sysineq} have only integer coefficients, they have a solution over $\R$ if and only if they have a solution over $\Q$.
Then, by their homogeneity, they have a solution over $\Q$ if and only if they have a solution over $\Z$.
Hence, the Systems~\eqref{eq:syslin}, \eqref{eq:syspos}, \eqref{eq:sysineq} have a solution over $\Z$, meaning $f_1 h_1 + \cdots + f_n h_n = 0$ has a solution $f_i = \sum_{j = 0}^{d_i} a_{ij} X^j, i = 1 \ldots, n,$ over $\N[X] \setminus \{0\}$.
\end{proof}

\propRtoU*
\begin{proof}
    Obviously a solution over $\Rp[X] \setminus \{0\}$ is a solution over $\U(\Rpp)$.

    For the other implication, we use P\'{o}lya's Theorem (Lemma~\ref{lem:polya}).
    Suppose $f_1 h_1 + \cdots + f_n h_n = 0$ has a solution $(f_1, \ldots, f_n)$ over $\U(\Rpp)$.
    Write $f_i = X^{c_i} \cdot F_i$ where $c_i \neq 0$ and $F_i \in \R[X]$ is such that $X \nmid F_i$.
    Since $X \nmid F_i$ we have $F_i(0) \neq 0$, we claim that $F_i(0) > 0$.
    In fact, if $F_i(0) < 0$, then by the continuity of $F_i$, there exists $\varepsilon > 0$ such that $F_i(\varepsilon) < 0$, but then $f_i(\varepsilon) = \varepsilon^{c_i} F_i(\varepsilon) < 0$, contradicting the fact that $f_i \in \U(\Rpp)$.
    Furthermore, one easily sees that $F_i(x) = \frac{f_i(x)}{x^{c_i}} > 0$ for all $x > 0$.
    So we have shown $F_i(x) > 0$ for all $x \geq 0$.
    
    We now show that for large enough $p \in \N$, the polynomials $\hat{f}_i \coloneqq (X + 1)^p \cdot f_i$ are all in $\Rp[X]$.
    Let $Y$ be a new variable, and for every $i$, let $G_i$ be the homogenization of $F_i$ using the variable $Y$.
    That is, $G_i = F_i(X/Y) \cdot Y^{\deg(F_i)}$.
    Since $F_i(x/y) > 0$ for all $x/y \geq 0$, we have $G_i(x, y) > 0$ for all $x \geq 0, y > 0$. 
    Whereas for $x > 0, y = 0$, $G_i(x, y)/x^{\deg(F_i)}$ is the leading coefficient of $F_i$.
    This is non-zero and thus must be positive because $\lim_{x \rightarrow \infty}F_i(x) > 0$.
    Therefore $G_i(x, y) > 0$ for $x > 0, y = 0$.
    
    We have thus shown $G_i(x, y) > 0$ for all $x \geq 0, y \geq 0, x+y > 0$.
    Applying P\'{o}lya's Theorem yields the existence of a $p_i \in \N$ such that $(X + Y)^{p_i} \cdot G_i \in \Rp[X, Y]$.
    Taking $Y = 1$ we dehomogenize $G_i$ and obtain $(X + 1)^{p_i} \cdot F_i \in \Rp[X]$.
    Let $p = \max\{p_1, \ldots, p_n\}$, then 
    \[
        \hat{f}_i = (X + 1)^{p} \cdot f_i = X^{c_i} \cdot (X + 1)^{p} \cdot F_i \in \Rp[X] \setminus \{0\}
    \]
    for all $i$.
    We have thus found the solution $(\hat{f}_1, \ldots, \hat{f}_n)$ over $\Rp[X] \setminus \{0\}$ for the equation $f_1 h_1 + \cdots + f_n h_n = 0$.
\end{proof}

\begin{rmk}\label{rmk:UvsW}
    Proposition~\ref{prop:RtoU} no longer holds if we replace $\U(\Rpp)$ with $\W(\Rpp) \setminus \{0\}$.
    For example, take $n = 2, h_1 = 1, h_2 = - (X-1)^2$.
    Then $f_1 = (X-1)^2, f_2 = 1$ is a solution over $\W(\Rpp) \setminus \{0\}$ of the equation $f_1 h_1 + f_2 h_2 = 0$.
    However, $f_1 h_1 + f_2 h_2 = 0$ does not admit a solution over $\Rp[X] \setminus \{0\}$.
    Indeed, any solution of $f_1 - f_2 \cdot (X-1)^2 = 0$ over $\R[X]$ must satisfy $f_1(1) = 0$, so $f_1$ cannot be in $\Rp[X] \setminus \{0\}$.
\end{rmk}

\wreathdec*
\begin{proof}
\begin{enumerate}
    \item For the Group Problem, by Proposition~\ref{prop:wrtopos} it suffices to decide whether there exists a set $S \subseteq I \times J$ satisfying $\pi_I(S) = I, \pi_J(S) = J$, such that the equation $\sum_{(i, j) \in S} f_{ij} h_{ij} = 0$ has a solution $(f_{ij})_{(i, j) \in S}$ over $\N[X] \setminus \{0\}$.
    By the homogeneity of the equation $\sum_{(i, j) \in S} f_{ij} h_{ij} = 0$, one can multiply all the Laurent polynomials $h_{ij}$ by a power of $X$ and suppose all $h_{ij} \in \Z[X]$.
    For every set $S \subseteq I \times J$ satisfying $\pi_I(S) = I, \pi_J(S) = J$, we can use Theorem~\ref{thm:decidability} to decide whether $\sum_{(i, j) \in S} f_{ij} h_{ij} = 0$ has a solution over $\N[X] \setminus \{0\}$.
    This shows the decidability of the Group Problem.
    \item The neutral element is in $\langle \mG \rangle$ if and only if a non-empty subset of $\mG$ generates a group (as a semigroup).
    This is because, if the product of a word $w \in \mG^+$ is the neutral element, then every element in the set $C$ of letters used in $w$ can be inverted in $\langle C \rangle$, so $\langle C \rangle$ is a group.
    Therefore, in order to decide whether the neutral element is in $\langle \mG \rangle$, it suffices to check for all subsets of $\mG$ whether they generate a group.
    This is decidable by the above result on the Group Problem.
\end{enumerate}
\end{proof}

\section{Algorithm for Theorem~\ref{thm:mainlocglob}}\label{app:alg}
\begin{algorithm}[h]
\caption{Deciding existence of solutions over $\N[X] \setminus \{0\}$ of the equation $f_1 h_1 + \cdots + f_n h_n = 0$.}
\label{alg:dec}
\begin{description}
\item[Input:] 
Polynomials $h_1, \ldots, h_n \in \Z[X]$.
\item[Output:] 
\textbf{True} or \textbf{False}.
\end{description}
\begin{enumerate}[(1)]
    \item
    Compute $d \coloneqq \gcd(h_1, \ldots, h_n)$ and divide all $h_i$ by $d$.
    \item Repeat the following:
    \begin{enumerate}
        \item If $h_i(0) > 0$ for all $i$, or $h_i(0) < 0$ for all $i$, return \textbf{False}.
        \item Else if $h_i(0) \geq 0$ for all $i$, or $h_i(0) \leq 0$ for all $i$, divide all the polynomials $h_i$ that satisfy $h_i(0) = 0$ by $X$.
        \item Else go to \ref{step:3}.
    \end{enumerate}    
    \item\label{step:3} Decide the truth of the existential sentence~\eqref{eq:sentreal} in the theory of reals.
    If \eqref{eq:sentreal} is true, return \textbf{False}, otherwise return \textbf{True}.
\end{enumerate}
\end{algorithm}

\section{Comparison with the Br\"ocker-Prestel local-global principle}\label{app:compare}
The original Br\"ocker-Prestel local-global principle (\cite[Theorem~8.13]{prestel2007lectures}) can be formulated as follows.

\begin{thrm}[Br\"ocker-Prestel local-global principle]\label{thm:Presori}
    Let $F$ be a formally real field, and $h_1, \ldots, h_n$ be non-zero elements of $F$.
    If the equation $f_1 h_1 + \cdots f_n h_n = 0$ has no non-trivial solution $(f_1, \ldots, f_n) \neq (0, \ldots, 0)$ over \emph{sums of squares} of $F$ (that is, over the set $S \coloneqq \{\sum_{i = 1}^k a_i^2 \mid a_i \in F\}$), then at least one of the following hold:
    \begin{enumerate}[(i)]
        \item $h_1, \ldots, h_n$ are all of the same sign in some archimedean ordering of $F$.
        \item $f_1 h_1 + \cdots f_n h_n = 0$ has no solution in the \emph{Henselization} of some real place of $F$.
    \end{enumerate}
\end{thrm}

For a definition of Henselizations of a formally real field, see \cite[Proposition~8.1]{prestel2007lectures}.

When applied to the field $F = \R(X)$, the Br\"ocker-Prestel local-global principle characterizes the absence of non-trivial solutions over sums of squares by condition (ii), since the field $\R(X)$ has no archimedean orderings.
Multiplying by the common denominator and using the fact that any element in $\W(\R)$ can be written as a sum of squares in $\R(X)$, Theorem~\ref{thm:Presori} also characterizes the absence of non-trivial solutions over $\W(\R)$.
However, when considering non-trivial solutions over $\U(\R)$ and $\U(B)$, the situation is quite different; and we now compare the proof of Theorem~\ref{thm:locglobpos} to Theorem~\ref{thm:Presori}.

    The proof of Br\"ocker-Prestel's original theorem starts with the definition of the pre-semicone
    \[
        P_1 \coloneqq \left\{\sum_{i = 1}^n f_i h_i, \text{ where } f_i \text{ are sum of squares of elements in } \R(X) \right\}.
    \]
    Since it considers solutions over sum of squares, this definition is straightforward.
    The definition of $P_0$ is our proof of Theorem~\ref{thm:locglobpos} is different and less straightforward.
    In our theorem, we are considering \emph{strictly} positive polynomials on $B \subseteq \R$, therefore we need to replace sum of squares with polynomials in $\U(B)$. 
    However, such a naive replacement does not work due to the requirement of a pre-semicone to be closed under multiplication of squares (unlike $\W(\R)$, the set $\U(B)$ is not closed under multiplication by squares).
    This is why we need to add the rational function $\frac{g}{G}$ in the definition of $P_0$ and use the fundamental theorem of algebra to guarantee closure under addition.

    Note that in order to guarantee the closure under addition of $P_0$, it is essential that we work in the \emph{univariate} polynomial ring $\R[X]$, so that two polynomials $g, g'$ having a common root implies $\gcd(g, g') \neq 1$.
    For example, this no longer holds in the bivariate polynomial ring $\R[X, Y]$.
    Therefore, even when supposing $\gcd(\frac{gG'}{dD}, \frac{g'G}{dD}) = 1$, we no longer have $\left(f_i \frac{gG'}{dD} + f'_i \frac{g'G}{dD}\right)(x, y) > 0$ in Equation~\eqref{eq:add}.
    Thus, for the field $\R(X, Y)$, the closure under addition of $P_0$ no longer holds, a contrast with the ``non-strict'' version $P_1$.

    The following step of extracting the valuation ring $A_{\R}^P$ from the semiordering $P$ appeared as part of the proof of the original theorem. 
    (The original theorem used the valuation ring $A_{\Q}^P$ instead, but they are in fact equivalent.)
    This is the main part where we drew inspiration from the original local-global principle.
    
    After extracting the valuation ring $A_{\R}^P$, our proof again diverges from that of the original theorem.
    Our new definition of $P_0$ allows us to enforce strict positivity, however it also takes away some convenient properties of the pre-semicone $P_1$ in the original theorem.
    Notably, we have $h_1 \in P_1$, allowing for a quick conclusion on the positivity of $h_1$ in Henselizations.
    Whereas for $P_0$, we do not have $h_1 \in P_0$ due to the strict positivity of the coefficients $f_i$.
    We compensate this by the analytic approach adopted in the second half of our proof, making use of the classification of real places of $\R(X)$ and the continuity of functions in $\R[X]$.
    This part is absent from the proof of the original theorem, which is purely algebraic and model theoretic.
\end{document}